\documentclass[11pt]{amsart}

 \usepackage{amsmath,amsfonts,amssymb, amsthm, color, hyperref, standalone, import}

 \usepackage[all, cmtip]{xy}
 
 \newtheorem{theorem}{Theorem}
 \newtheorem{lemma}{Lemma}
 \newtheorem{clm}{Claim}
 \newtheorem{prop}{Proposition}
 \newtheorem{cor}{Corollary}
 
 \newcommand{\tpi}{2\pi i}
 \newcommand{\hfsq}[1]{\displaystyle\frac{1}{2\sqrt{#1}}}
 \newcommand{\fsq}[1]{\displaystyle\frac{1}{\sqrt{#1}}}
 \newcommand{\meta}[1]{\mathrm{SO}(#1)_2}
 \newcommand{\roi}{\mathcal{O}}
 \newcommand{\sgen}{\sigma}
 \newcommand{\tgen}{\tau}
 \newcommand{\mcg}[1]{\textrm{MCG}(\Sigma_{#1})}
 \newcommand{\modgp}{\textrm{SL}(2, \mathbb{Z})}
 
 \newcommand{\gl}[1]{\textrm{GL}(#1)}
 \newcommand{\pgl}[1]{\textrm{PGL}(#1)}
 \newcommand{\rep}{\rho_1}
 
 \newcommand{\tw}[1]{\theta_{#1}}
 \newcommand{\tx}{\psi}
 \newcommand{\vrow}[1]{1 & \tw #1 & \tw #1^2 & \cdots & \tw #1^r\\}
 \newcommand{\cob}[2]{#1^{-1} #2 #1}
 \newcommand{\vect}[1]{[1, {\tw 1}^{#1}, \tw 2 ^{#1},\cdots, \tw r ^{#1}]^t}

 \newcommand{\kzero}{H}
 
 \newcommand{\obj}[1]{\textbf{#1}}
   \newcommand{\sorig}{\left[\begin{array}{ccccc}
 		\displaystyle{\hfsq p}&\displaystyle{\hfsq p}& \displaystyle{\fsq p\cdot a^t} &\displaystyle{\frac{1}{2}} &\displaystyle{\frac{1}{2}} \\
 		\\
 		\displaystyle{\hfsq p}&\displaystyle{\hfsq p}& \displaystyle{\fsq p\cdot a^t} &\displaystyle{-\frac{1}{2}}&\displaystyle{-\frac{1}{2}}\\
 		\\
 		\displaystyle{\fsq p \cdot a}&\displaystyle{\fsq p\cdot a}&\displaystyle A&\displaystyle{0_{r\times 1}} &\displaystyle{0_{r\times 1}} \\
 		\\
 		\displaystyle \frac{1}{2} &\displaystyle{ -\frac{1}{2}}& \displaystyle{0_{1\times r}}&\displaystyle{\frac{1}{2}}&\displaystyle{-\frac{1}{2}} \\
 		\\
 		\displaystyle{\frac{1}{2}}&\displaystyle{-\frac{1}{2}}& \displaystyle{0_{1\times r}}&\displaystyle{-\frac{1}{2}}&\displaystyle{\frac{1}{2}} \\
 	\end{array}
 	\right]}
 \newcommand{\torig}{\begin{bmatrix}
     1&&&&&&\\
     &1&&&&&\\
     &&\tw 1&&&&\\
     &&&\ddots&&&\\
     &&&&\tw r&\\
     &&&&&\tx&\\
     &&&&&&-\tx\\
   \end{bmatrix}}

 \newcommand{\sump}{\sum_{k = 1}^r}
 \newcommand{\srep}{\rep(\sgen)}
 \newcommand{\trep}{\rep(\tgen)}
 \newcommand{\ubasis}{\mathcal{B}_U}
 \newcommand{\sblock}{\left[
	\begin{array}{ccc|ccc}
		0 & 1 & 0 &   &&                 \\
		1 & 0 & 0 &   &0_{3\times (r+1)} \\
		0 & 0 & 1 &   &&               \\
		\hline
		&&&&&\\
		& && \fsq p  &\fsq p \cdot a^t\\
		&0_{(r+1)\times 3} &&&&\\
		& &&\tfsq p \cdot a &   A    \\
	\end{array}\right]}

 \newcommand{\tfsq}[1]{\displaystyle\frac{2}{\sqrt{#1}}}
 
 \newcommand{\tblock}{\left[
     \begin{array}{ccc|cccc}
       1&0   &0   &&&&\\
       0&0   &\tx^2 &&&&\\
       0&1 &0   &&&&\\
       \hline
        &&        &1&&&\\
        &&        &&\tw 1&&\\
        &&        &&&\ddots&\\
        &&        &&&&\tw r
     \end{array}
   \right]}
 \newcommand{\spanc}{\textrm{span}_\mathbb{C}}
 
 \newcommand{\sprime}
{
	\begin{bmatrix}
		\fsq p &\fsq p\cdot a^t\\
		\tfsq p\cdot a &A \\
	\end{bmatrix}
}

 \newcommand{\tprime}
 {
   \begin{bmatrix}
     1&&&\\
     &\tw 1&&\\
     &&\ddots&\\
     &&&\tw r    
   \end{bmatrix}
 }
 
 \newcommand{\dmat}
 {
   \begin{bmatrix}
     \displaystyle\frac{1}{2} &  &       & \\
     & 1&       & \\
     &  & \ddots& \\
     &  &       &1
   \end{bmatrix}
 }
 
 \newcommand{\vdm}
 {
   \begin{bmatrix}
     1 & 1 & 1 & \cdots & 1\\
     \vrow 1
     \vrow 2
     \vdots & \vdots &\vdots &&\vdots\\
     \vrow r
   \end{bmatrix}
 }
 
 \newcommand{\smat}
{
	\begin{bmatrix}
		\fsq p &\tfsq p\cdot a^t\\
		\fsq p \cdot a &A \\
	\end{bmatrix}
}
 
 \newcommand{\iotafunc}
 {
   \begin{cases}
     1, & \mbox{ if }\ p\equiv 1\mod 4,\\
     i, & \mbox{ if }\ p\equiv 3\mod 4.
   \end{cases}
 }

 \newcommand{\sotp}{\sum_{l = 1}^r}
 \newcommand{\szttp}{\sum_{l = 0}^{2r}}
 \newcommand{\zp}[1]{\zeta^{#1}}
 \newcommand{\ffield}{\mathbb{Z}/p\mathbb{Z}}
 \newcommand{\jac}[2]{\displaystyle\Big(\frac{#1}{#2}\Big)_J}
 \newcommand{\vinv}[1]{(V^{-1})_{#1}}
 \newcommand{\dprod}[2]{\displaystyle\prod_{#1 = 1}^{#2}}

    \newcommand{\lie}[1]{\mathfrak{#1}}
    \newcommand{\irr}{\mathrm{Irr}(\meta{p})}
    \newcommand{\oz}{\obj{Z}}
    \newcommand{\ox}{\obj{X}}
    \newcommand{\oy}{\obj{Y}}
    \newcommand{\oo}{\obj{1}}

  \newcommand{\heis}[1]{\mathcal{H}_{#1}}

  \newcommand{\fmodgp}{\mathrm{SL}(2, \ffield)}

  \newcommand{\helt}[3]{
    \begin{bmatrix}
      1 & #2 & #3\\
      0 & 1  & #1\\
      0 & 0  & 1\\
    \end{bmatrix}
  }

  \newcommand{\slelt}[4]{
    \begin{bmatrix}
      #1 & #2\\
      #3 & #4
    \end{bmatrix}  
    }

  \newcommand{\tvc}[2]{
    \begin{bmatrix}
      #1\\
      #2
    \end{bmatrix}
    }  

  \newcommand{\cchar}{\varphi}

  \newcommand{\hrep}{\pi}

  \newcommand{\funcsp}[1]{\mathcal{L}^2(#1)}  

  \newcommand{\wrep}[1]{W_{#1}} 

  \newcommand{\evensp}{E^{even}}

  \newcommand{\oddsp}{E^{odd}}

\title{On modular group representations associated to $\mathrm{SO}(p)_2$-TQFTs}
\author{Yilong Wang}

\date{}

\begin{document}

\address{Department of Mathematics, The Ohio State University, Columbus, OH 43210, USA}
\email{wang.3003@osu.edu}

\begin{abstract}
In this paper, we prove that for any odd prime larger than 3, the modular group  representation associated to the $\mathrm{SO}(p)_2$-TQFT can be defined over the ring of integers of a cyclotomic field. We will provide explicit integral bases. In the last section, we will relate these representations to the Weil representations over finite fields.
\end{abstract}
\maketitle

\section{Introduction}\label{sec1}
  Let $p$ be an odd prime. According to \cite{MR1091619}, to each modular tensor category, we can associate a Reshetikhin-Turaev TQFT. This TQFT not only gives rise to quantum invariants of 3-manifolds, but also to a series of projective representations of the mapping class groups $\mcg{g}$ of closed oriented surfaces $\Sigma_g$ of genus $g$. In particular, in genus one, we get a projective representation of the modular group $\mcg{1} = \modgp$.
  
  A systematic way to construct modular categories is to consider the representation theory of quantum groups at roots of unity. The TQFT representations of mapping class groups arising from such modular categories are finite-dimensional and can be defined over a cyclotomic field $\mathbb{Q}(\zeta)$ where $\zeta$ is a root of unity. As a result, the corresponding quantum invariants are elements of $\mathbb{Q}(\zeta)$.
  
  The first integrality result was obtained by Murakami \cite{MR1277059,MR1307078}, who showed that the SU(2)- and SO(3)-invariants are algebraic integers when the order of $\zeta$ is prime. More precisely, those invariants are elements of the ring of integers of $\mathbb{Q}(\zeta)$, namely, $\mathbb{Z}[\zeta]$. The result was reproved in \cite{MR1434653}, generalized to all classical Lie types in \cite{MR1662260,MR1697388}, then to all Lie types by Le \cite{MR1953323}. These results helped us relate the quantum invariants to other invariants such as the Casson invariant \cite{MR1277059,MR1307078} and the Ohtsuki series \cite{MR1297897,MR1374199,MR1953323}.

  A natural question to ask then is whether one can define the whole TQFT over $\mathbb{Z}[\zeta]$, or at least can one define the representations of the mapping class groups over $\mathbb{Z}[\zeta]$. If that is the case, we can get more information about these representations. For example, in \cite{MR2382862} studied the Frohman Kania-Bartoszynska ideal \cite{MR1857120} using the integral SO(3)-TQFT, and this method is supposed to be generalized to other integral TQFTs. Another possible  application of integral TQFTs is that we can reduce these TQFTs by the natural reduction map $\mathbb{Z}[\zeta]\to\ffield$ to get the so-called $p$-modular TQFTs.  $p$-modular TQFTs have rich connections to topological information of 3-manifolds, such as the Casson-Lescop invariant, the Milnor torsion, see for example, \cite{MR2002607}.  We may also answer questions such as the finiteness of the images of these representations by the integrality results.
  
  For the SO(3)-TQFT, Gilmer, Masbaum and van Wamelen first constructed integral bases for genus one and two \cite{MR2059432}. Then Gilmer and Masbaum generalized the result to arbitrary genus in \cite{MR2382862}, hence completed the construction of the integral SO(3)-TQFT.
  
  In this paper, we will focus on the integrality properties of the $\meta{p}$-TQFTs for $p \geq 5$, which comes from the representation theories of quantum groups associated to the Lie algebra $\lie{so}(p)$ at certain roots of unity. These TQFTs emerge as important objects in the context of topological quantum computing \cite{MR3215577}, and as interesting examples of classical Lie type quantum groups themselves. We will establish the integrality of them in genus one by proving

  \begin{theorem}\label{thm1}
  	Suppose $p \geq 5$ is an odd prime. Then the genus one mapping class group  representation given by the $\meta{p}$-TQFT can be defined over $\roi$, where
  	\begin{equation}
  	\roi =
  	\begin{cases} 
  	\mathbb{Z}[\zeta_p, i] = \mathbb{Z}[\zeta_{4p}], & \mbox{if } p \equiv 3\  (mod\ 4) \\ 
  	\mathbb{Z}[\zeta_p], & \mbox{if } p \equiv 1\ (mod\ 4).
  	\end{cases}
  	\end{equation} 
  \end{theorem}

  In the proof, we give an explicit integral basis as the authors of \cite{MR2059432} did for the SO(3)-TQFT. As a byproduct, we show that a part of the genus one mapping class group representation of $\modgp$ factors through a part of the Weil representation of $\fmodgp$, then conclude  as a corollary that the image of the $\meta{p}$-TQFT representation in genus one is finite. This confirms the theorem by \cite{MR2725181} saying that the $\modgp$ TQFT representation given by any modular category is finite.
  
  The paper is organized as follows. In \hyperref[sec2]{Section 2}, we give a quick review of preliminaries on the TQFT representation and the $\meta{p}$-TQFT.  In particular, we give explicit matrix presentations of the $\modgp$ representation under a fixed basis. In \hyperref[sec3]{Section 3}, we use number theoretic tools to construct a new basis for the representation space and show that it is indeed integral. In \hyperref[sec4]{Section 4}, we briefly recall the definition of the Weil representation of $\fmodgp$. We prove in \hyperref[eqweil]{Theorem 2} that a part of the genus one TQFT representation factors through a part of the Weil representation. As a result, we prove that the image of the $\modgp$ representation is finite.
  
  \textit{Notations and conventions.} In the discussion below, we will assume that $p\geq 5$ is an odd prime. Let $r = \frac{p-1}{2}$. Let $\zeta_n = e^{\frac{2\pi i}{n}}$ be an $n$-th root of unity. We will let $\roi$ be as in \hyperref[thm1]{Theorem 1}, and use $\roi^\times$ to denote the group of units of $\roi$. It is a well-known fact that $ \sqrt{p}\in\roi $, hence $ \frac{1}{\sqrt{p}}\in\mathbb{Q}(\zeta_p)$ or $\mathbb{Q}(\zeta_{4p})$ depending on $p$. Let $\ast^t$  denote the transpose of $\ast$, and let $Id_n$ stands for the $n\times n$ identity matrix. We call a representation \textbf{integral} if the matrix coefficients of the representation are in $\roi$ with respect to certain basis. Sometimes we will also call a matrix with entries in $\roi$ integral.

  \textbf{Acknowledgments.} The author would like to thank his advisor Professor Thomas Kerler for his guidance and many stimulating discussions. The author would like to thank Professor James W. Cogdell for helping the author understanding the Weil representation. The author is also grateful to Professor Patrick Gilmer, Professor Eric Rowell and Professor Zhenghan Wang for helpful discussions and their advice.

\section{Preliminaries}\label{sec2}
  In this section, we briefly recall the definition of the modular category $\meta{p}$ and the Reshetikhin-Turaev TQFT associated to it. For more details, the readers are referred to \cite{MR1091619, MR3215577}.

  The $\meta{p}$ modular category is the unitary modular tensor category obtained from the representation theory of the quantum group $U_q(\lie{so}(p))$, where $q = e^{\frac{\pi i}{2p}}$. It has $(r + 4)$ simple objects, which we will label as $ \irr = \{\obj{1}, \obj{Z}, \obj{Y}_1, ..., \obj{Y}_r, \obj{X}, \obj{X}'\}$. Here $\obj{1}$ is the tensor unit.  The fusion rules can be completely determined by the following listed ones:
  
\begin{equation}\label{fusionrules}
  \begin{array}{ll}
  \obj{Z}\otimes \obj{Z}&\cong\ \ \obj{1},\\
  \\
  \obj{Z}\otimes \obj{X}&\cong\ \ \obj{X}',\\
    \\
  \obj{Z}\otimes \obj{Y}_j&\cong\ \ \obj{Y}_j, \ \ \forall j = 1,\cdots, r,\\
  \\
  \ox\otimes\ox&\cong\ \ \oo\oplus\bigoplus_{j = 1}^r \oy_j,\\
  \\
  \ox\otimes\ox'&\cong\ \ \oz\oplus\bigoplus_{j = 1}^r \oy_j,\\
  \\
  \ox\otimes\oy_j&\cong\ \ \ox\oplus\ox', \ \ \forall j = 1,\cdots, r,\\
  \\
  \oy_j\otimes\oy_j&\cong\ \ \oo\oplus\oz\oplus\oy_{\min\{2j,\ m-2j\}}, \ \ \forall j = 1,\cdots, r,\\
  \\
  \oy_j\otimes\oy_k&\cong\ \ \oy_{|j - k|}\oplus\oy_{\min\{j+k,\ m - j - k\}}, \forall 1\leq j\ ,k\leq r, \ \ j \neq k.
  \end{array}
\end{equation}
  
  In the following, we will let
  \begin{equation}\label{kzero}
  \kzero = \mathrm{span}_\mathbb{C} \{\obj{1}, \obj{Z}, \obj{Y}_1, ..., \obj{Y}_r, \obj{X}, \obj{X}'\},
  \end{equation}
  and we will view $\irr$ as the fixed basis of $\kzero$. 
  
  To any modular tensor category, we can construct a Reshetikhin-Turaev TQFT. The TQFT is, roughly speaking, a tensor functor from a suitably defined cobordism category to the category of finite-dimensional vector spaces. The cobordism category has closed oriented surfaces as objects and 3-manifolds bounding two such surfaces as morphisms. The tensor structure on the cobordism category is the disjoint union, and the tensor structure on the category of vector spaces is the tensor product over the ground field.
  
  In particular, for any orientation preserving diffeomorphism $\Psi$ of a surface $\Sigma_g$ of genus $g$, the image of the mapping cylinder of $\Psi$ under the TQFT functor is a linear automorphism of the vector space associated to $\Sigma_g$. It is unique up to scalar multiples and is invariant under isotopy. As a result, the TQFT gives rise to projective representations of the mapping class groups of surfaces of genus $g$ for each $g\in \mathbb{Z}_{>0}$.
  
  As $\meta{p}$ is a modular tensor category, we can of course consider the TQFT associated to it. In the genus one case, we have $\mcg 1 = \modgp$ and we give the representation explicitly in terms of the generators of $\modgp$ as follows.
  
  Let 
  $\sgen =
  \begin{bmatrix}
  0 & -1\\
  1 & 0
  \end{bmatrix}
  $ and  
  $\tgen =
  \begin{bmatrix}
  1 & 1\\
  0 & 1
  \end{bmatrix}
  $
  be generators of $\modgp$. Let $A$ be the $(r\times r)$-matrix with entries given by
  
  \begin{equation}
  \label{eq:centralblock}
  A_{jk} = \displaystyle \frac{2}{\sqrt{p}}\cos(\frac{2\pi jk}{p}) = \displaystyle \frac{1}{\sqrt{p}}(\zeta^{jk}+\zeta^{-jk})
  \end{equation}
  for all $1 \leq j, k\leq r$.
  
  Let 
  \begin{equation}\label{rowvec}
  a =
  \begin{bmatrix}
  1\\
  \vdots\\
  1
  \end{bmatrix}
  \end{equation}
  be an $(r\times 1)$-dimensional vector, and let
  
  \begin{equation}
  \label{eq:twistx}
  \tx = \zeta_8^{r} = e^{\tpi\cdot\frac{r}{8}}.
  \end{equation}
  Note that $\tx\notin\roi$, but $\tx^2\in\roi$.
  
  From now on, we will suppress the subscript of $\zeta_p$ by simply writing $\zeta$,  while letting
  
  \begin{equation}
  \label{eq:twist}
  \tw j = \zeta^{rj^2} = e^{\frac{2\pi i}{p} rj^2} =e^{\frac{\tpi}{p}\cdot\frac{j(p  - j)}{2}}, \ \ \ \forall j\in \mathbb{Z}.
  \end{equation}
  The projective representation of $\modgp$ derived from the $\meta{p}$-TQFT is given by (see, for example, \cite{MR2832261}):
  
  \begin{equation}
  \label{eq:rep}
  \rep: \modgp \to \pgl{\kzero},
  \end{equation}
  
  \begin{equation}
  \label{eq:Sorig}
  \rep(\sgen) =\sorig,
  \end{equation}
  and
  
  \begin{equation}
  \label{eq:Torig}
  \rep(\tgen) =\torig.
  \end{equation}
  Here $b_{\mu\times \nu}$ in a matrix $M$ is understood as a block of $M$ whose dimension is $\mu\times \nu$ and whose entries are all $b$.
  
  Now we are ready to proceed to the proof of \hyperref[thm1]{Theorem \ref*{thm1}}.

\section{Proof of \texorpdfstring{\hyperref[thm1]{Theorem 1}}{}}\label{sec3}
  
  To prove the theorem, we will give an explicit change of basis matrix $W \in \gl{\kzero}$ so that $W^{-1}\rep (\sgen)W$ and $W^{-1}\rep (\tgen)W$ have entries in $\roi$. We will find $W$ in several steps. First, we decompose $ \kzero $ into a direct sum of two invariant subspaces in \hyperref[lm1]{Lemma \ref*{lm1}}, and reduce the problem to \hyperref[clm1]{Claim \ref*{clm1}}. We then investigate properties of the column vectors of the representation after the  change of basis proposed in \hyperref[clm1]{Claim \ref*{clm1}}. We will prove integrality of one of the column vector in \hyperref[prop2]{Proposition \ref*{prop2}}. Finally, we will prove the integrality of the rest of the column vectors by proving \hyperref[vectclm]{Claim \ref*{vectclm}}.
  
  Let
  \begin{equation}
  \label{eq:Umat}
  U =
  \left[
  \begin{array}{ccccc}
  1 & 0           & 0  & 1           &  \\
  &&&& 0_{2\times r}\\
  -1& 0           & 0  & 1           &  \\
  &&&&\\
  & 0_{r\times 4} &    & & Id_r \\
  &&&&\\
  0 & 1           & \tx  & 0           & \\
  &&&& 0_{2\times r}\\
  0 & 1           & -\tx & 0           & \\
  \end{array}
  \right].
  \end{equation}
  
  \begin{lemma}\label{lm1}
  	After the change of basis by $U$, $\kzero$
  	splits into a direct sum of two invariant subspaces $\kzero \cong H_1\oplus
  	H_2$, with $\rep|_{H_1}$ integral under the new basis.
  \end{lemma}
  
  \begin{proof}
  	From the form of $U$, it is easy to see that the new basis corresponding to $U$  is $\ubasis = \{\obj{1}-\obj{Z}, \obj{X} + \obj{X}', \tx(\obj X - \obj X'), \obj 1  + \obj Z, \obj Y_1, ..., \obj Y_r\}$. To determine the matrix coefficients of $\rep (\sgen)$ and $\rep(\tgen)$ after the change of basis, we simply have to calculate  how the two linear operations act on the new basis vectors. Then we will write the  resulting vectors as linear combinations of vectors in $\ubasis$.
  	
  	By (\ref{eq:Sorig}), we have
  	\begin{equation}
  	\label{eq:sact1}
  	\begin{array}{rcl}
  	\rep(\sgen)(\obj 1)&=& \displaystyle\hfsq p\obj 1 + \hfsq p\obj Z + \fsq p\sump\obj Y_k + \frac{1}{2}\obj X + \frac{1}{2}\obj X',\\
  	\\
  	\rep(\sgen)(\obj Z)&=& \displaystyle\hfsq p\obj 1 + \hfsq p\obj Z + \fsq p\sump\obj Y_k - \frac{1}{2}\obj X - \frac{1}{2}\obj X',\\
  	\\
  	\rep(\sgen)(\obj Y_j)&=& \displaystyle\fsq p\obj 1 + \fsq p\obj Z + \sump A_{ kj}\obj Y_k, \ \forall j = 1, ..., r,\\
  	\\
  	\rep(\sgen)(\obj X)&=& \displaystyle\frac{1}{2}\obj 1 - \frac{1}{2}\obj Z + \frac{1}{2}\obj X - \frac{1}{2}\obj X',\\
  	\\
  	\rep(\sgen)(\obj X')&=& \displaystyle\frac{1}{2}\obj 1 - \frac{1}{2}\obj Z - \frac{1}{2}\obj X + \frac{1}{2}\obj X'.\\
  	\end{array}
  	\end{equation}
  	So the action of $\rep(\sgen)$ on the new basis vectors (written as linear  combinations of them) is given by
  	\begin{equation}
  	\label{eq:sact2}
  	\begin{array}{rcl}
  	\srep(\obj 1 - \obj Z) &=& \obj X + \obj X',\\
  	\\
  	\srep(\obj X + \obj X') &=& \obj 1 - \obj Z,\\
  	\\
  	\srep(\tx(\obj X -\obj X')) &=& \tx(\obj X - \obj X'),\\
  	\\
  	\srep(\obj 1 + \obj Z) &=& \fsq p(\obj 1 + \obj Z) + \tfsq p\sump\obj Y_k,\\
  	\\
  	\srep(\obj Y_j) &=& \fsq p(\obj 1 + \obj Z) + \sump A_{kj}\obj Y_k.
  	\end{array}
  	\end{equation}
  	Therefore, the linear map $\srep$ has the following matrix presentation in the new  basis $\ubasis$:
  	\begin{equation}
  	\label{eq:sucob}
  	U^{-1}\srep U =\sblock.
  	\end{equation}
  	
  	By a similar argument, we have
  	\begin{equation}
  	\label{eq:tact2}
  	\begin{array}{rcl}
  	\trep(\obj 1 - \obj Z) &=& \obj 1 - \obj Z,\\
  	\\
  	\trep(\obj X + \obj X')&=& \tx (\obj X - \obj X'),\\
  	\\
  	\trep(\tx(\obj X - \obj X'))&=& \tx^2 (\obj X + \obj X'),\\
  	\\
  	\trep(\obj 1 +\obj Z) &=& \obj 1 + \obj Z,\\
  	\\
  	\trep(\obj Y_j) &=& \tw j\obj Y_j.
  	\end{array}
  	\end{equation}
  	Therefore, under $\ubasis$, $\trep$ has matrix presentation
  	\begin{equation}
  	\label{eq:tucob}
  	U^{-1}\trep U = \tblock.
  	\end{equation}
  	The empty slots in the matrix are considered as 0-matrices of suitable size.
  	
  	It is easy to see, either from the actions of $\srep$ and $\trep$ or from the  block form of their matrix presentations in the basis $\ubasis$, that they  preserve the subspaces $H_1 = \spanc\{\obj 1 - \obj Z, \obj X + \obj X', \tx(\obj  X - \obj{X}')\}$ and the subspace $H_2 = \spanc\{\obj{1} + \obj{Z}, \obj Y_ 1, ..., \obj Y_r\}$ of $\kzero$. So we have
  	
  	\begin{equation}\label{decompofkz}
  	\kzero \cong H_1\oplus H_2.
  	\end{equation}
  	In addition, the matrix coefficients of $U^{-1}\srep U$  and $U^{-1}\trep U$ restricted  to $H_1$ are in $\roi$.
  \end{proof}

  Given \hyperref[lm1]{Lemma \ref*{lm1}}, we just have to find a change of basis for  the $(r+1)$-dimensional subspace $H_2$ so that $\rep$
  restricted to $H_2$ is integral. For convenience, we introduce the following notations: 
  
  \begin{equation}
  \label{eq:sprime}
  S' = \rep|_{H_2}(\sgen) = \sprime
  \end{equation}
  and
  \begin{equation}
  \label{eq:tprime}
  T'= \rep|_{H_2}(\tgen) = \tprime.
  \end{equation}
  
  Instead of $S'$ and $T'$, we would prefer to work with their transposes. We define
  
  \begin{equation}
  \label{eq:thes}
  S := (S')^t = \cob{D}{S'} = \smat
  \end{equation}
  and 
  \begin{equation}
  \label{eq:thet}
  T:=(T')^t = \cob{D}{T'} = T' = \tprime.
  \end{equation}
  Here
  
  \begin{equation}
  \label{eq:dmat}
  D = \dmat.
  \end{equation}

  \hyperref[thm1]{Theorem \ref*{thm1}} now follows from the claim below:
  
  \begin{clm}\label{clm1}
  	Let $V$ be the following Vandermonde matrix
  	\begin{equation}
  	\label{eq:vandermonde}
  	V = \vdm.
  	\end{equation}
  	Then $\cob{V}{S}$ and $\cob{V}{T}$ are both integral.
  \end{clm}
  
  \begin{proof}[Proof of Theorem 1]
  	The change of basis matrix $(Id_3 \oplus DV)$ makes the block matrices
  	$\cob{U}{\rep(\sgen)}$ and $\cob{U}{\rep(\tgen)}$ integral. Hence $W = U(Id_3\oplus DV)$
  	is the desired change of basis matrix.
  \end{proof}
  
  To prove the claim, we need the following property of
  $SV$. (Convention: in the following discussions, we will index the
  matrix entries from 0, and recall that by definition $\tw{0} = 1$.)
  
  \begin{prop}\label{prop1}
  	The $(j,k)$-th matrix coefficient of $SV$ is given by
  	\begin{equation}
  	\label{eq:svprop}
  	(SV)_{jk} =
  	\begin{cases}
  	\sqrt{p}, & \mbox{ if } j = k = 0,\\
  	0, & \mbox{ if } 1\leq j \leq r, \mbox{ and } k = 0,\\
  	\displaystyle \Big(\frac{rk}{p}\Big)_J\cdot \iota(p)\cdot \tw{j}^{-\frac{1}{k} }, &
  	\mbox{ if } k\geq 1.
  	\end{cases}
  	\end{equation}
  	Here $\displaystyle\Big(\frac{\ast}{\ast}\Big)_J$ stands for the Jacobi symbol  and
  	\begin{equation}
  	\label{eq:delta}
  	\iota(p) = \iotafunc
  	\end{equation}
  \end{prop}
  
  \begin{proof}
  	\textbf{Case 1.} When $j = k = 0$, a direct computation shows that 
  	\begin{equation}
  	\label{eq:jk0}
  	(SV)_{00} = \fsq p + \tfsq p \times r = \fsq{p} \times p= \sqrt{p}.
  	\end{equation}
  	
  	\textbf{Case 2.} When $1\leq j \leq r$, and $k = 0$, we have,
  	\begin{equation}
  	\label{eq:jzero}
  	\begin{array}{ccl}
  	(SV)_{j0} &=& \fsq{p} + \sum_{l = 1}^r A_{jl}\times 1\\
  	\\
  	&=& \fsq{p} + \sum_{l = 1}^r \fsq{p} (\zeta^{jl} + \zeta^{-jl})\\
  	\\
  	&=& \fsq{p} \sum_{l = 0}^{2r} \zeta^{jl}\\
  	\\
  	&=& 0.
  	\end{array}
  	\end{equation}
  	The third equality results from the fact that $-l \equiv p - l\mod p$. The last equality stands on the fact that for any $1\leq j\leq r$, $\zeta^j$ is  an $p$-th root of unity, hence is a solution to the minimal polynomial $\Phi_p(x)  = 1 + x + \cdots + x^{2r}$ (recall that by assumption $p$ is an odd prime).
  	
  	\textbf{Case 3.} When $1\leq k\leq r$, we have
  	\begin{equation}
  	\begin{array}{ccl}
  	(SV)_{jk}
  	&=& \fsq{p} + \sotp\fsq{p} (\zeta^{jl} + \zeta^{-jl})\times \tw l ^k\\
  	\\
  	&=& \fsq{p} + \fsq{p}\sotp (\zp{jl} + \zp{-jl})\times\zp{rkl^2}\\
  	\\
  	&=& \fsq{p}\szttp\zp{jl+rkl^2}\\
  	\\
  	&=&\fsq{p}\szttp\zp{rk(l^2+\frac{j}{rk} l)}.
  	\end{array}
  	\end{equation}
  	Note that by assumption, $1\leq k\leq r$, hence $\frac{j}{rk}$ is well-defined in  the finite field $\ffield$. Letting $\gamma = \frac{j}{2rk}\in\ffield$, we can continue  our calculation as follows:
  	\begin{equation}
  	\begin{array}{ccl}
  	\fsq{p}\szttp\zp{rk(l^2 + 2\gamma l)}
  	&=& \fsq{p}\szttp\zp{rk(l+\gamma)^2 - rk\gamma^2}\\
  	\\
  	&=& \fsq{p}\zp{-rk\gamma^2}\szttp\zp{rk(l+\gamma)^2}.
  	\end{array}
  	\end{equation}
  	Hence by the quadratic Gauss sum formula, we have
  	\begin{equation}
  	\begin{array}{ccl}
  	(SV)_{jk}
  	&=& \fsq{p}\times\zp{-rk\gamma^2}\times\jac{rk}{p}\times\iota(p)\times\sqrt{p}\\
  	\\
  	&=& \jac{rk}{p}\times\iota(m)\times\zp{-rk\gamma^2}.
  	\end{array}
  	\end{equation}
  	Note that
  	\begin{equation}
  	4r^2 - 1 = (2r+1)(2r-1) = p(2r-1)\equiv 0\mod p.
  	\end{equation}
  	Therefore,
  	\begin{equation}
  	k\gamma^2 = \displaystyle\frac{j^2}{4r^2k} \equiv \displaystyle\frac{j^2}{k}\mod p,
  	\end{equation}
  	and consequently,
  	\begin{equation}
  	\label{eq:svlast}
  	(SV)_{jk} = \jac{rk}{p}\times\iota(p)\times\zp{(rj^2)\times(-\frac{1}{k})} = \jac{rk}{p}\times\iota(p)\times\tw j ^{-\frac{1}{k}},
  	\end{equation}
  	as desired.
  \end{proof}
  
  To proceed further, let's recall some basic facts in number theory.
  
  \begin{lemma}
  	\label{numberthry}
  	Let $\epsilon = (-1)^r\times \zeta^{-\frac{r(r+1)}{2}}\in\roi^\times$. Then
  	\begin{equation}
  	\label{eq:3}
  	p = 2r+1 = \epsilon\dprod{k}{r}(1 - \zeta^k)^2.
  	\end{equation}
  \end{lemma}
  
  \begin{proof}
  	Recall that 
  	\begin{equation}
  	\label{eq:4}
  	\Phi_p(x) = 1 + x + \cdots + x^{2r} = \dprod{l}{2r} (x - \zeta^l).
  	\end{equation}
  	Putting $x = 1$, we have 
  	\begin{equation}
  	\label{eq:5}
  	\begin{array}{ccl}
  	p\ \ =\ \ 2r+1
  	&=& \dprod{l}{2r} (1 - \zeta^l)\\
  	\\        
  	&=& \dprod{k}{r} (1 - \zeta^{k})\times(1 - \zeta^{-k})\\
  	\\
  	&=& \dprod{k}{r} (1 - \zeta^k)\times\zeta^{-k}\times(\zeta^k - 1)\\
  	\\
  	&=& \dprod{k}{r} (-\zeta^{-k})\times\dprod{k}{r}(1 - \zeta^k)^2\\
  	\\
  	&=& (-1)^r\times\zeta^{-\frac{r(r+1)}{2}}\times\dprod{k}{r}(1 - \zeta^k)^2\\
  	\\
  	&=& \epsilon\dprod{k}{r}(1 - \zeta^k)^2.
  	\end{array}
  	\end{equation}
  \end{proof}
  
  Note that $\epsilon^{\frac{1}{2}}$ is also in $\roi^\times$. Indeed, by definition,  $(-1)^{\frac{r}{2}}$ is a power of $\tx^2$ and $\frac{-r(r+1)}{4}$ is well-defined  in $\ffield$. Note also that for any integers $\alpha, \beta$ such that $g.c.d(\alpha, p) = g.c.d.(\beta, p) = 1$, we have
  \begin{equation}
  \label{eq:unit}
  \displaystyle\frac{1 - \zeta^\alpha}{1 - \zeta^{\beta}} \in \roi^\times.
  \end{equation}
  This is because in $\ffield$, we can write $\alpha$ as a multiple of $\beta$, so  the quotient in (\ref{eq:unit}) becomes a sum of elements in $\roi$, so it is in $\roi $. On the other hand, we can write $\beta$ as a multiple of $\alpha$, then the  inverse of the quotient in (\ref{eq:unit}) is also a sum of elements in $\roi$, hence  in $\roi$.
  
  Combining \hyperref[numberthry]{Lemma \ref*{numberthry}} and the above observation,  we have
  \begin{cor}\label{cor1}
  	\label{sqrtunit}
  	\begin{equation}
  	\label{eq:sqrtunit}
  	\sqrt{p} = \dprod{k}{r}(1 - \tw k)\times u,
  	\end{equation}
  	and $u\in \roi^\times$.
  \end{cor}
  
  \begin{proof}
  	By \hyperref[numberthry]{Lemma \ref*{numberthry}}, we know that
  	\begin{equation}
  	p = \epsilon\dprod{k}{r}(1 - \zeta^k)^2.
  	\end{equation}
  	Then
  	\begin{equation}
  	\sqrt{p} = \epsilon^{\frac{1}{2}}\dprod{k}{r}(1 - \zeta^k),
  	\end{equation}
  	where $\epsilon^{\frac{1}{2}}\in\roi^\times$ as shown above. Also by the  observation above, we have, for any $1\leq k\leq 2r$,
  	\begin{equation}
  	\eta_k = \displaystyle\frac{1 - \zeta^k}{1 - \tw k}  = \displaystyle\frac{1 - \zeta^k}{1 - \zeta^{rk^2}}\in \roi^\times.
  	\end{equation}
  	Letting $u = \epsilon^{\frac{1}{2}}\times\dprod{k}{r}\eta_k$, we have
  	\begin{equation}
  	\sqrt{p} = \epsilon^{\frac{1}{2}}\dprod{k}{r}(1 - \zeta^k) = \dprod{k}{r}(1 - \tw k)\times(\epsilon^{\frac{1}{2}}\times\dprod{k}{r}\eta_k) = \dprod{k}{r}(1 - \tw k)\times u.
  	\end{equation}
  	Note that $u$, as product of elements in $\roi^\times$, is in $\roi^\times$.
  \end{proof}
  
  \begin{prop}\label{prop2}
  	The 0-th column of $\cob{V}{S}$ is a vector in $\roi^{r+1}$.
  \end{prop}
  
  \begin{proof}
  	By \hyperref[prop1]{Proposition \ref*{prop1}}, for any $j$, we have
  	\begin{equation}
  	(\cob{V}{S})_{j0} = \sum_{l = 0}^r \vinv{jl}(SV)_{l0} =\vinv{j0}\times\sqrt{p}.
  	\end{equation}
  	To prove the proposition, we simply have to show that $\vinv{j0}\times\sqrt{p}\in\roi$.
  	
  	By definition, we have
  	\begin{equation}
  	\label{eq:invmat}
  	V\cdot\vinv{} = Id_{r+1}.
  	\end{equation}
  	In other words, for any $0\leq k\leq r$,
  	\begin{equation}
  	\label{eq:invmatentry}
  	\sum_{j = 0}^r \tw{k}^j\times\vinv{j0} = \delta_{k,0},
  	\end{equation}
  	where $\delta_{\ast,\ast}$ is the Kronecker delta function. Consider the  polynomial
  	\begin{equation}
  	\label{eq:invpoly}
  	P_0(x) = \sum_{j = 0}^r \vinv{j0}\times x^j.
  	\end{equation}
  	By (\ref{eq:invmatentry}), we  have
  	\begin{equation}
  	\label{eq:values}
  	P_0(\tw{0}) = 1,\ \ P_0(\tw{k}) = 0,\ \ k = 1,...,r.
  	\end{equation}
  	Therefore, by the Lagrangian interpolation formula, we have
  	\begin{equation}
  	\label{eq:interpolation}
  	\displaystyle P_0(x) = \sum_{j = 0}^r \vinv{j0}\times x^j = \prod_{n = 1}^r\frac {x - \tw{n}}{1 - \tw{n}}.
  	\end{equation}
  	
  	By comparing coefficients, we can write down explicit formulas for $\vinv{j0}$. But what is  more important here is that 
  	\begin{equation}
  	\vinv{j0}\times\prod_{n = 1}^r(1 - \tw{n}) \in\roi,
  	\end{equation}
  	since it is a  coefficient of the integral polynomial $\prod_{n = 1}^r(x - \tw{n})$. On the other hand, by \hyperref[sqrtunit]{Corollary \ref*{cor1}}, we have $\sqrt{p} = \prod_{n = 1}^r (1 - \tw{n})\times u$ for some unit $u\in \roi^{\times}$, hence
  	\begin{equation}
  	\vinv{j0}\times\sqrt{p} = \vinv{j0}\times\prod_{n = 1}^r (1 - \tw{n})\times u\in  \roi.
  	\end{equation}
  \end{proof}

  By \hyperref[prop2]{Proposition \ref*{prop2}}, we are left to show
  that the $l$-th column vector of $SV$ for $1\leq l\leq r+1$ and all the
  column vectors of $TV$ have the property that after multiplying them
  $V^{-1}$ from left we get vectors in $\roi^{r+1}$. 
  
  In light of \hyperref[prop1]{Proposition \ref*{prop1}}, we have the
  following observation:
  
  \begin{lemma}\label{vectlem}
  	The $l$-th column vector of $SV$ for $1\leq l\leq r+1$ and all the
  	column vectors of $TV$ are, up to a scalar multiplication by $\pm i$ or $\pm 1$, of  the form $\vect{j}$ for some $0\leq j\leq 2r$.
  \end{lemma}

  \begin{proof}
  	This is a direct result of \hyperref[prop1]{Proposition \ref*{prop1}} and the definition  of $T$.
  \end{proof}
  
  Hence we reduce our problem to the problem of showing
  
  \begin{clm}\label{vectclm}
  	The vectors in \hyperref[vectlem]{Lemma \ref*{vectlem}}, after
  	being multiplied by $V^{-1}$ from left, become vectors in
  	$\roi^{r+1}$.
  \end{clm}
  
  We will need the following lemma.
  
  \begin{lemma}\label{keylm}
  	Let $f(x) = (x - x_0)(x - x_1)(x - x_2)\cdots (x - x_r) = x^{r+1} + a_1x^r + \cdots + a_rx + a_{r + 1}$. Then the matrix
  	\begin{equation}
  	C = \begin{bmatrix}
  	0           & 1     & 0          & 0          &\cdots    & 0     \\
  	0           & 0     & 1          & 0          & \cdots   & 0     \\
  	0           & 0     & 0          & 1          & \cdots   & 0     \\
  	\vdots      &\vdots &\vdots      &\vdots      &          &\vdots \\
  	0           & 0     & 0          & 0          & \cdots   & 1     \\
  	-a_{r+1}      &-a_r   & -a_{r-1}    &-a_{r-2}     &\cdots    &-a_1
  	\end{bmatrix}
  	\end{equation}
  	has $[1, x_i, x_i^2,\cdots, x_i^r]^t$ as eigenvectors corresponding to
  	eigenvalues $x_i$ for any $0\leq i\leq r$.
  \end{lemma}
  
  \begin{proof}
  	For any $0\leq i\leq r$, we have
  	\begin{equation}
  	\label{eq:1}
  	C \begin{bmatrix}
  	1\\
  	x_i\\
  	x_i^2\\
  	\vdots\\
  	x_i^r\\
  	\end{bmatrix} = 
  	\begin{bmatrix}
  	x_i\\
  	x_i^2\\
  	x_i^3\\
  	\vdots\\
  	-a_{r+1} - a_rx_i-\cdots - a_1x_i^r
  	\end{bmatrix}.
  	\end{equation}
  	
  	Since $f(x_i) = x_i^{r+1} + a_1x_i^r + \cdots + a_rx_i + a_{r + 1} =
  	0$, we have $x_i^{r+1} = -a_{r+1} - a_rx_i-\cdots - a_1x_i^r$. Hence
  	\begin{equation}
  	\label{eq:2}
  	C
  	\begin{bmatrix}
  	1\\
  	x_i\\
  	x_i^2\\
  	\vdots\\
  	x_i^r
  	\end{bmatrix}
  	=
  	\begin{bmatrix}
  	x_i\\
  	x_i^2\\
  	x_i^3\\
  	\vdots\\
  	x_i^{r+1}
  	\end{bmatrix}
  	= x_i
  	\begin{bmatrix}
  	1\\
  	x_i\\
  	x_i^2\\
  	\vdots\\
  	x_i^r
  	\end{bmatrix}.
  	\end{equation}
  \end{proof}
  
  \begin{cor}
  	$\cob{V}{T}$ has entries in $\roi$. Consequently, $\cob{V}{T^j}$ has
  	entries in $\roi$ for every $0\leq j\leq 2r$. In particular, their
  	first columns are vectors in $\roi^{r+1}$.
  \end{cor}
  
  \begin{proof}
  	Let $x_k = \tw{k}$ in \hyperref[keylm]{Lemma \ref*{keylm}}. Then we have the  corresponding polynomial $h(x) = (x - 1)(x - \tw{1})\cdots(x - \tw{r}) = x^{r+1} +  b_1x^r+\cdots +b_{r+1}$ with its companion matrix
  	\begin{equation}
  	C_h =
  	\begin{bmatrix}
  	0           & 1     & 0          & 0          &\cdots    & 0     \\
  	0           & 0     & 1          & 0          & \cdots   & 0     \\
  	0           & 0     & 0          & 1          & \cdots   & 0     \\
  	\vdots      &\vdots &\vdots      &\vdots      &          &\vdots \\
  	0           & 0     & 0          & 0          & \cdots   & 1     \\
  	-b_{r+1}      &-b_r   & -b_{r-1}    &-b_{r-2}     &\cdots    &-b_1
  	\end{bmatrix}.
  	\end{equation}
  	By \hyperref[keylm]{Lemma \ref*{keylm}}, $V^t$ diagonalizes $C_h$:
  	\begin{equation}
  	\label{eq:diagonalize}
  	\cob{(V^t)}{C_h} = T.
  	\end{equation}
  	Taking the transpose of both sides, we have
  	\begin{equation}
  	\label{eq:transpose}
  	V(C_h)^t\vinv{} = T.
  	\end{equation}
  	Note that $T$ is diagonal, then
  	\begin{equation}
  	\label{eq:conclusion}
  	\cob{V}{T} = (C_h)^t.
  	\end{equation}
  	As the $\roi$ is a ring, $b_k\in\roi$ for all $k$. Hence all entries in $C_h$ are in $\roi$, so is its transpose therefore $\cob{V}{T}$. The rest of the  corollary follows immediately.
  \end{proof}
  
  Observing that the first columns of $T^jV$ correspond exactly to the
  vectors in \hyperref[vectclm]{Claim \ref*{vectclm}}, we can conclude that
  \hyperref[vectclm]{Claim \ref*{vectclm}} is true. As a result,
  \hyperref[clm1]{Claim \ref*{clm1}}, as well as \hyperref[thm1]{Theorem
  	1}, is true.

\section{Weil representation over finite fields}\label{sec4}
  In Section 3 of \cite{MR2183960}, the genus one representation of $\modgp$ associated to the $\mathrm{SO}(3)$-TQFT for a fixed odd prime $p\geq 5$ (in the sense of \cite{MR1362791}) was considered, where the authors identified the representation with the odd part of the Weil representation of $\fmodgp$. Here we will prove a result in some sense ``dual'' to that in \cite{MR2183960}. Namely, for the fixed prime $p$, a factor of $\rep$ factors through the even part of the Weil representation of $\fmodgp$.
  
  To clarify the above paragraph, let us briefly recall the definition of the Weil representation over finite fields. The basic idea is to realize elements in $\fmodgp$ as intertwining operators of the Heisenberg representation of the Heisenberg group, which will be defined below. There is a vast amount of research on the Weil representations, and we will only extract some essential ingredients of the representation of $\fmodgp$ here. The interested readers are referred to \cite{MR0460477}. 
  
  Fix an odd prime $p\geq 5$. We start by looking at a group called the Heisenberg group $\heis{p}$, defined by

  \begin{equation}\label{Heisgp}
  \heis{p} = \Bigg\{\helt{x}{y}{z},\  x,\  y,\  z\in \ffield\Bigg\}.
  \end{equation}
  Here the group multiplication is the matrix multiplication. Considering the embedding
  \begin{equation}\label{center}
  \ffield \to \heis{p}, \ \ z\mapsto\helt{0}{0}{z},
  \end{equation}
  we can view the group  $\heis{p}$ as a central extension of $\ffield$. More precisely, we have a short exact sequence 
  \begin{equation}\label{ses}
  0 \to \ffield \to \heis{p} \to (\ffield)^2 \to 0.
  \end{equation}
  The quotient map is given by
  \begin{equation}\label{quot} 
  \heis{p}\to (\ffield)^2 , \ \ \helt{x}{y}{z}\mapsto \tvc{x}{y}.
  \end{equation}

  With a suitable choice of section to the quotient map above, it is not difficult to show that the defining action of $\fmodgp$ on $(\ffield)^2$, 
  
  \begin{equation}
  \slelt{a}{b}{c}{d}\tvc{x}{y} = \tvc{ax+by}{cx+dy},   \ \mathrm{where}\ \slelt{a}{b}{c}{d}\in\fmodgp,\ \tvc{x}{y}\in (\ffield)^2,
  \end{equation}
  can be lifted to $\heis{p}$, and the lifted action is trivial on the center $Z(\heis{p})$ of $\heis{p}$.
  
  Let $\funcsp{\ffield}$ denote the space of complex-valued functions on $\ffield$. It is easily seen that $\dim(\funcsp{\ffield}) = p$. Given any irreducible central character $\cchar : \ffield \to \mathbb{C}$, we can define a representation $\hrep_\cchar : \heis{p} \to \gl{\funcsp{\ffield}}$ by
  
  \begin{equation}\label{heisrep}
  \left(\hrep_\cchar \left(\helt{x}{y}{z}\right)(f)\right)(a) = \cchar (-xa + z)f(a - y),
  \end{equation}
  for any $\helt{x}{y}{z}\in\heis{p}$ and $f\in\funcsp{\ffield}$. 
  
  Since $\hrep_\cchar$ is $ p-$dimensional, by the representation theory of finite groups, it is either a direct sum of $p$ 1-dimensional representations or irreducible. However, in the first case, $\hrep_\cchar|_{Z(\heis{p})}$ should be trivial, which contradicts to our assumption on $\cchar$.
  
  By Theorem 3.1 of \cite{MR2722604}, if two irreducible representations of $\heis{p}$ coincide on the center $Z(\heis{p})$, then they are equivalent. Now let $\cchar$ be any nontrivial irreducible central character. For any $\alpha \in \fmodgp$, consider the representation $\hrep_\cchar\circ\alpha$, a $p$-dimensional representation of $\heis{p}$ with the property 
  
  \begin{equation}
  (\hrep_\cchar\circ\alpha)|_{Z(\heis{p})} = \cchar = \hrep_\cchar|_{Z(\heis{p})}.
  \end{equation} 
  By a similar argument as above, we know that $\hrep_\cchar\circ\alpha$ is also irreducible. Hence $\hrep_\cchar\circ\alpha$ is equivalent to $\hrep_\cchar$, i.e., there is an intertwining operator (unique up to scalar by Schur's lemma), denoted by $W_\cchar(\alpha) \in\gl{\funcsp{\ffield}}$ such that the diagram
  
  \begin{equation}\label{inttw}
  \xymatrix{
  	\funcsp{\ffield} \ar[rr]^{\hrep_\cchar(h)} \ar[d]_{W_\cchar(\alpha)} && \funcsp{\ffield} \ar[d]^{W_\cchar(\alpha)}\\
  	\funcsp{\ffield} \ar[rr]^{\hrep_\cchar(\alpha(h))} && \funcsp{\ffield}
  }
  \end{equation}
  commutes for all $h \in \heis{p}$.
  
  If we consider the class of $W_\cchar(\alpha)$ in the projective general linear group $\pgl{\funcsp{\ffield}}$ instead of $W_\cchar(\alpha)\in \gl{\funcsp{\ffield}}$, we can eliminate the scaling ambiguity and get a well-defined projective representation of $\fmodgp$ (by abuse of notation this map is also denoted by $\wrep{\cchar}$):
  
  \begin{equation}\label{weil}
  \wrep\cchar: \fmodgp \to \pgl{\funcsp{\ffield}}.
  \end{equation}
  We call this projective representation the Weil representation of $\fmodgp$ (with respect to $\cchar$).
  
  \textit{Remark.} We may omit the word ``projective'' when it does not cause confusions, and we will present an element in $\pgl{\funcsp{\ffield}}$ by one of its representatives in $\gl{\funcsp{\ffield}}$.

  Again,  let 
  $\sgen =
  \begin{bmatrix}
  0 & -1\\
  1 & 0
  \end{bmatrix}
  $ and 
  $\tgen =
  \begin{bmatrix}
  1 & 1\\
  0 & 1
  \end{bmatrix}
  $ be the generators of $\modgp$. Their reductions mod $p$ generate $\fmodgp$. By an abuse of notation, we will not distinguish $\sgen$ and $\tgen$ from their reductions. For $j \in \ffield$, let $f_j: \ffield\to\ffield$ be the $j-$th Kronecker delta function defined by
  
  \begin{equation}\label{deltafunc}
  f_j(x) = \delta_{j,x}, \ \ \forall x\in \ffield.
  \end{equation}
  The set $\{f_j| j \in \ffield\}$ is a basis of $\funcsp{\ffield}$, which we fix in the following.
  
  Now, to describe the Weil representation with respect to a nontrivial character $\cchar$, it suffices to give the matrices of $\wrep{\cchar}(\sgen)$ and $\wrep{\cchar}(\tgen)$ under the fixed basis defined above. It is not difficult to compute that
  
  \begin{equation}\label{weils}
  \wrep{\cchar}(\sgen) = 
  \begin{bmatrix}
  1 & 1           & 1              & \cdots & 1\\
  1 & \cchar(1)   & \cchar(2)      & \cdots & \cchar(p-1)\\
  1 & \cchar(2)   & \cchar(4)      & \cdots & \cchar(2(p-1))\\
  \vdots&\vdots   & \vdots         & \cdots & \vdots\\
  1 & \cchar(p-1) & \cchar(2(p-1)) & \cdots & \cchar((p-1)^2)\\
  \end{bmatrix}
  \end{equation}
  and that
  \begin{equation}\label{weilt}
  \wrep{\cchar}(\tgen) = 
  \begin{bmatrix}
  1 &&&&&\\
  & \cchar(-\frac{1^2}{2}) &&&&\\
  && \cchar(-\frac{2^2}{2}) &&&\\
  &&& \cchar(-\frac{3^2}{2}) &&\\
  &&&& \ddots &\\
  &&&&& \cchar(-\frac{(p-1)^2}{2})\\
  \end{bmatrix}.
  \end{equation}
  As before, $\frac{1}{2}$ is understood as the multiplicative reciprocal of $2$ in $\ffield$.

  Note that this representation is reducible. Indeed, it is easy to see that the $\mathbb{C}$-span of $\{f_k+f_{p-k}| k = 0,1,\cdots,r\}$ and $\{f_k - f_{p - k} | k = 0, 1, \cdots, r\}$ are two invariant subspaces. If we denote the former vector space by $\evensp$ and the latter by $\oddsp$, then we have a decomposition of representation spaces $\funcsp{\ffield} \cong \evensp \oplus\oddsp$.
  
  We are mainly interested in the restriction of the Weil representation on the even subspace $\evensp$. By (\hyperref[weils]{\ref{weils}}) and (\hyperref[weilt]{\ref{weilt}}), we have
  
  \begin{equation}\label{wevens}
  \wrep{\cchar}^{even}(\sgen) = 
  \wrep{\cchar}|_{\evensp}(\sgen) = 
  \begin{bmatrix}
  1             & a^t\\
  2\cdot a & B  
  \end{bmatrix}.
  \end{equation}
  Here $B$ is an $r\times r$-matrix with entries given by
  
  \begin{equation}\label{bmat}
  B_{jk} = \cchar(jk) + \cchar(-jk), \ \ \forall j, k = 1, \cdots , r.
  \end{equation}
  In addition, we have 
  
  \begin{equation}\label{wevent}
  \wrep{\cchar}^{even}(\tgen) = 
  \wrep{\cchar}|_{\evensp}(\tgen) = 
  \begin{bmatrix}
  1 &&&&&\\
  & \cchar(-\frac{1^2}{2}) &&&\\
  && \cchar(-\frac{2^2}{2}) &&\\
  &&& \ddots &\\
  &&&& \cchar(-\frac{r^2}{2})\\
  \end{bmatrix}. 
  \end{equation}

  If we choose the special character $\cchar: \ffield \to \mathbb{C}$ defined by
  
  \begin{equation}\label{specialchar}
  \cchar(j) = \zeta^j, 
  \end{equation}
  we will have $\sqrt{p} A = B$ and 
  
  \begin{equation}\label{charandtheta}
  2r \equiv -1\mod{p} \ \Rightarrow\  r\equiv -\frac{1}{2}\mod{p}\ \Rightarrow\  \cchar(-\frac{j^2}{2}) = \zeta^{rj^2} = \theta_j.
  \end{equation}
  Therefore,
  
  \begin{equation}\label{wsrou}
  \wrep{\cchar}^{even}(\sgen) = 
  \displaystyle\sqrt{p}
  \begin{bmatrix}
  \displaystyle\frac{1}{\sqrt{p}}& \displaystyle\frac{1}{\sqrt{p}}\cdot a^t\\
  \displaystyle\frac{2}{\sqrt{p}}\cdot a & A  
  \end{bmatrix},
  \end{equation}
  and
  \begin{equation}\label{wtrou}
  \wrep{\cchar}^{even}(\tgen) = 
  \begin{bmatrix}
  1 &&&\\
  & \theta_1 &&\\
  && \ddots &\\
  &&& \theta_r
  \end{bmatrix}. 
  \end{equation}

  Recall from previous sections that $\kzero \cong H_1\oplus H_2$ and that $H_2$ is an $(r+1)$-dimensional vector space. We can then identify $H_2$ and $\evensp$ via
  
  \begin{equation}\label{identification}
  \obj{1} + \obj{Z} \leftrightarrow 2f_0,\ \ \obj{Y}_j \leftrightarrow f_j+f_{-j},\ \forall j = 1, ..., r.
  \end{equation}

  With all the ingredients ready, we now state the theorem of this section:
  
  \begin{theorem}\label{eqweil}
  	Let $\cchar$ be chosen as in Equation (\hyperref[specialchar]{\ref{specialchar}}), then the restriction of the $\meta{p}$-TQFT representation of $\modgp$ to $H_2$, $\rep|_{H_2}$, factors through $\wrep{\cchar}^{even}$, the even part of the Weil representation of $\fmodgp$ associated to $\cchar$. In other words, we have the following commutative diagram:
  	\begin{equation}\label{tqftweil}
  	\xymatrix{
  		\modgp \ar[rr]^{\rep|_{H_2}} \ar[d]_{\mod p} &&
  		\pgl{H_2} \ar[d]^{\cong}\\
  		\fmodgp \ar[rr]^{\wrep{\cchar}^{even}} &&
  		\pgl\evensp
  	}.
  	\end{equation}
  \end{theorem}
  
  \begin{proof}
  	By Equations (\hyperref[eq:sprime]{\ref{eq:sprime}}), (\hyperref[eq:tprime]{\ref{eq:tprime}}), (\hyperref[wsrou]{\ref{wsrou}}) and (\hyperref[wtrou]{\ref{wtrou}}), we know that $\rep|_{H_2}$ and $\wrep{\cchar}^{even}$ are only different by a scalar multiple, hence as projective representations, they are the same.
  \end{proof}
  
  We immediately have the following corollaries:
  
  \begin{cor}
  	The image of $\rep|_{H_2}$ is finite.
  \end{cor}
  
  From equations (\hyperref[eq:tucob]{\ref{eq:tucob}}) and (\hyperref[eq:sucob]{\ref{eq:sucob}}), it is easy to see that the image of $\rep|_{H_1}$ can be viewed as a subgroup of the permutation group of the finite set $\{\pm 1, \pm i\}\times\{\obj{1} - \obj{Z}, \obj{X} \pm \obj{X'}\}$, so $\rep|_{H_1}(\modgp)$ is also finite. Hence, together with the above corollary, we have:
  
  \begin{cor}
  	The image of $\rep$ is finite.
  \end{cor}
  
  \textit{Remark.} The above corollary is a special case of the famous finiteness result obtained by \cite{MR2725181}.

\bibliographystyle{alpha}
\bibliography{/Users/Gohan/Dropbox/work/Papers/Reference.bib}

\end{document}